\documentclass[12pt]{amsart}

\usepackage{latexsym,amsmath,amssymb,amsthm,amsfonts}

\usepackage{bbding}

\usepackage[numeric]{amsrefs}

\newcommand{\R}{{\mathbb{R}}}
\newcommand{\eps}{\varepsilon}

\renewcommand{\P}{{\mathcal{P}}}
\newcommand{\Q}{{\mathcal{Q}}}
\newcommand{\x}{{\boldsymbol{x}}}
\newcommand{\bxi}{{\boldsymbol{\xi}}}
\newcommand{\bmu}{{\boldsymbol{\mu}}}
\renewcommand{\u}{{\boldsymbol{u}}}
\renewcommand{\v}{{\boldsymbol{v}}}
\newcommand{\n}{{\boldsymbol{n}}}
\newcommand{\y}{{\boldsymbol{y}}}
\renewcommand{\v}{{\boldsymbol{v}}}
\newcommand{\z}{{\boldsymbol{z}}}
\newcommand{\zero}{{\boldsymbol{0}}}

\newcommand{\wt}{\widetilde}
\newcommand{\Vol}{{\rm Vol}}
\newcommand{\diam}{{\rm diam}}

\newcommand{\tx}{{\wt\x}}

\def\be{\begin{equation}}
\def\ee{\end{equation}}

\DeclareMathOperator*{\esssup}{ess\,sup}

\newtheorem{theorem}{Theorem}[section]
\newtheorem{lemma}[theorem]{Lemma}
\newtheorem{conjecture}[theorem]{Conjecture}
\newtheorem{corollary}[theorem]{Corollary}
\newtheorem{observation}[theorem]{Observation}
\theoremstyle{remark}
\newtheorem{remark}[theorem]{Remark}
\newtheorem{example}[theorem]{Example}
\newtheorem{definition}[theorem]{Definition}

\renewcommand{\phi}{\varphi}

\numberwithin{equation}{section}

\title{On Nikol'skii inequalities for domains in $\R^d$}
\author{Z.\ Ditzian}
\address{Department of Mathematical and Statistical Sciences, University of Alberta, Edmonton, AB, T6G2G1, Canada}
\email{zditzian@shaw.ca, zditzian@math.ualberta.ca}

\author{A.\ Prymak}
\address{Department of Mathematics, University of Manitoba, Winnipeg, MB, R3T2N2, Canada}
\email{prymak@gmail.com}
\thanks{The second author was supported by NSERC of Canada Discovery Grant RGPIN 342218.}



\begin{document}

\begin{abstract}
Nikol'skii inequalities for various sets of functions, domains and weights will be discussed. Much of the work is dedicated to the class of algebraic polynomials of total degree $n$ on a bounded convex domain $D$. That is, we study $\sigma:= \sigma(D,d)$ for which
\[
\|P\|_{L_q(D)}\le c n^{\sigma(\frac1p-\frac1q)}\|P\|_{L_p(D)},\quad 0<p\le q\le\infty,
\]
where $P$ is a polynomial of total degree $n$. We use geometric properties of the boundary of $D$ to determine $\sigma(D,d)$ with the aid of comparison between domains. Computing the asymptotics of the Christoffel function of various domains is crucial in our investigation. The methods will be illustrated by the numerous examples in which the optimal $\sigma(D,d)$ will be computed explicitly.
\end{abstract}

\maketitle

\section{Introduction}

Nikol'skii inequalities have many uses and were investigated in many articles (see~\cite{Di14}, \cite{DiPrRMMJ}, \cite{DiTi}, \cite{Ka}, \cite{Ne-To} and~\cite{Ni} for example). In most texts on approximation (see~\cite[Theorem~2.6, p.~102]{De-Lo} and~\cite[4.3.6, p.~130]{Tr-Be} for example) one finds the classical cases of such inequalities given for $0<p\le q\le \infty$ by
\be\label{1.1}
\|P\|_{L_q[-1,1]}\le c n^{2(\frac1p -\frac1q)}\|P\|_{L_p[-1,1]}, \quad P\in{\rm span}\,\{1,\dots,x^{n-1}\}
\ee
and
\be\label{1.2}
\|\tau\|_{L_q[-\pi,\pi]}\le c n^{(\frac1p-\frac1q)}\|\tau\|_{L_p[-\pi,\pi]}, \quad \tau\in{\rm span}\,\{e^{ikx}:|k|<n\}.
\ee

In other investigations different sets of functions, different domains and different measures were explored.

Our main goal will be to establish the relation between $\|P\|_{L_q(D)}$ and $\|P\|_{L_p(D)}$ for many bounded domains in $\R^d$ where $P\in\P_n=\P_{n,d}$ is the set of algebraic polynomials of total degree $\le n$. For this purpose we develop methods to study $\sigma:= \sigma(D,d)$ in relation to $D$ and $d$ for which
\be\label{1.3}
\|P\|_{L_q(D)}\le c n^{\sigma(\frac1p-\frac1q)}\|P\|_{L_p(D)},\quad 0<p\le q\le\infty,
\ee
where $P\in\P_n$. Here and elsewhere in this paper when $p=q$ one has the obvious inequality. We do not wish to exclude the correct case $p=q$.

On our way we will obtain results of more general form for different weights or sets of functions which we use later. We hope these more general results will be useful in the future.

In Section~\ref{sect-gen} we obtain some general results relating $L_p$ and $L_q$ norms for some finite dimensional space of functions. In Section~\ref{sect-chris} we describe the relation of Nikol'skii inequalities with estimates of the Christoffel functions. In Section~\ref{sect-christ-ball-cube} we use the results on the unit ball and the cube to deduce results on $\sigma$ (see~\eqref{1.3}) for many domains. In Section~\ref{sect-upper} we introduce the new method of extension to obtain upper estimates for the behaviour of $\sigma$ that fit more domains. In Section~\ref{sect-lower} we obtain additional results to establish the lower estimate of $\sigma$. In Section~\ref{sect-examples} we examine many examples in which we make use of Sections~\ref{sect-gen}, \ref{sect-chris} and \ref{sect-christ-ball-cube}. An effort is made to establish sharp estimates of $\sigma$, that is, to show that for some polynomials $P$ of total degree $n$, some $c_1$ and some $p$ and $q$, $p<q$ (mainly $2$ and $\infty$),
\[
\|P\|_{L_q(D)}\ge c_1 n^{\sigma(\frac1p-\frac1q)}\|P\|_{L_p(D)}
\]
with $\sigma=\sigma(D,d)$ of~\eqref{1.3}. The behavior of the Christoffel function on $l_\alpha$ balls in $\R^d$, $1\le \alpha\le\infty$, is computed in Section~\ref{sect-lp}. Further applications of our technique are given in Section~\ref{sect-sharp}.

We will use the notation $\varphi(n)\approx\psi(n)$ to indicate that $c^{-1}\psi(n)\le\varphi(n)\le c\psi(n)$ with some positive $c$ independent of $n$.

\section{General results}\label{sect-gen}

In this section we obtain some general results. We begin with the following basic result/observation.

\begin{theorem}\label{dt2.1}
Suppose $D$ is a domain in $\R^d$, $\mu$ a positive measure on $D$, and $S$ a subspace of essentially bounded functions on $D$. Suppose also that for some $p$, $0<p<\infty$, and some $M>0$,
\be\label{d2.1}
\|f\|_{L_\infty(D,\mu)}\le M\|f\|_{L_p(D,\mu)},\quad\text{for all } f\in S.
\ee
Then, for any $f\in S$,
\be\label{d2.2}
\|f\|_{L_\infty(D,\mu)}\le M^{p/q}\|f\|_{L_q(D,\mu)}
\quad\text{for }0<q\le p
\ee
and
\be\label{d2.3}
\|f\|_{L_r(D,\mu)}\le M^{(\frac1q-\frac1r)p}\|f\|_{L_q(D,\mu)}
\quad\text{for }0<q\le p
\text{ and }r\ge q,
\ee
where
\be\label{d2.4}
\|f\|_{L_q(D,\mu)}=\left(\int_D|f|^qd\mu\right)^{1/q}
\quad\text{for }0<q<\infty\quad\text{and}\quad
\|f\|_{L_\infty(D,\mu)}=\inf_{\mu(A)=0}\sup(|f(\x)|:\x\in D\setminus A).
\ee
\end{theorem}
\begin{proof}
We have
\[
\|f\|_{L_\infty(D,\mu)}\le M \left(\int_D|f|^pd\mu\right)^{1/p}
\le M \|f\|_{L_\infty(D,\mu)}^{(p-q)/p} \left(\int_D|f|^qd\mu\right)^{1/p}.
\]
Therefore,
\[
\|f\|_{L_\infty(D,\mu)}^{q/p}\le M \left(\int_D|f|^qd\mu\right)^{1/p},
\]
which implies~\eqref{d2.2}. Using~\eqref{d2.2}, we now have
\[
\|f\|_{L_r(D,\mu)}\le \|f\|_{L_\infty(D,\mu)}^{1-\frac qr}\|f\|_{L_q(D,\mu)}^{\frac qr}
\le M^{\frac pq(1-\frac qr)}\|f\|_{L_q(D,\mu)}.
\]
\end{proof}

\begin{remark}
In the above proof we used the condition~\eqref{d2.1} for a single and exactly the same function $f$ for which we obtained~\eqref{d2.2} and~\eqref{d2.3}. Further discussion will often be relevant to the validity of~\eqref{d2.1} for a certain subspace $S$, so we included the notation $S$ here for convenience.
\end{remark}

\begin{remark}
If $\mu$ is nonsingular and $d\mu=w(\x)dx$ with a function $w(\x)>0$ a.e. in $D$, $\|f\|_{L_\infty(D,\mu)}=\|f\|_{L_\infty(D)}=\esssup_{\x\in D}|f(\x)|$. One observes that when $\mu$ is an atomic measure on $D$, that is $\mu(\x)=\sum\mu(\{\bxi_k\})\delta_{\bxi_k}(\x)$ (where $\delta_\bxi$ is the point mass at $\bxi$), $\|f\|_{L_\infty(D,\mu)}=\sup|f(\bxi_k)|$ and $\|f\|_{L_p(D,\mu)}=(\sum|f(\bxi_k)|^p\mu(\{\bxi_k\}))^{1/p}$, $0<p<\infty$, in which case Theorem~\ref{dt2.1} still applies.
\end{remark}

Recall that $\P_n=\P_{n,d}$ denotes the space of algebraic polynomials of total degree $\le n$ in $d$ variables.

\begin{corollary}\label{dc2.2}
Suppose $D\subset\R^d$ is a bounded domain, $\mu$ a finite positive measure on $D$ and suppose for some fixed $n$
\be\label{d2.5}
\|\phi\|_{L_\infty(D,\mu)}
\le c n^{\sigma/p}\|\phi\|_{L_p(D,\mu)}
\quad\text{for all }\phi\in\P_n,
\ee
where $c>0$ is independent of $n$. Then for $0<q\le p$ and $r\ge q$
\be\label{d2.6}
\|\phi\|_{L_r(D,\mu)}\le c_1 n^{\sigma(\frac1q-\frac1r)}\|\phi\|_{L_q(D,\mu)}, \quad \phi\in\P_n,
\ee
where $c_1=c_1(c,r,q)$ is a positive constant depending only on $c$, $r$ and $q$.\\
If in addition we assume~\eqref{d2.5} for all $n$, we have for $0<q\le r\le\infty$
\be\label{d2.6'}
\|\phi\|_{L_r(D,\mu)}\le c_2 n^{\sigma(\frac1q-\frac1r)}\|\phi\|_{L_q(D,\mu)}, \quad \phi\in\P_n,
\ee
where $c_2=c_2(c,r,q)$ is a positive constant depending only on $c$,  $r$ and $q$.
\end{corollary}

\begin{proof}
The inequality~\eqref{d2.6} is an immediate corollary of Theorem~\ref{dt2.1} as the boundedness of $D$ guarantees that $\phi\in\P_n$ is in $L_\infty(D)$ or $L_\infty(D,\mu)$ and as $\mu$ is a finite measure on $D$, $\phi\in\P_n$ implies $\phi\in L_s(D,\mu)$ for all $0<s<\infty$.

To prove~\eqref{d2.6'} we choose an integer $s$ such that $sp\ge q$ (the case $q=r=\infty$ is obvious). As $\phi\in\P_n$ implies $\phi^s\in\P_{ns}$, we have
\[
\|\phi^s\|_{L_\infty(D,\mu)}\le c(ns)^{\sigma/p}\|\phi^s\|_{L_p(D,\mu)},
\]
and hence
\[
\|\phi\|_{L_\infty(D,\mu)}\le c^{1/s}(ns)^{\sigma/(ps)}\|\phi\|_{L_{ps}(D,\mu)},
\]
which implies~\eqref{d2.6'}
for 
$r\ge q$.
\end{proof}

\begin{remark}
The above result remains valid if the sequence of polynomial subspaces $\{\P_n\}$ is replaced with any sequence $\{\Q_n\}$ of subspaces of $L_\infty(D,\mu)$ satisfying the property that $\phi\in\Q_n$ implies $\phi^s\in\Q_{ns}$ for all positive integers $s$. Useful examples of such sequences of subspaces may include radial polynomials, even polynomials (of certain degree), and others.
\end{remark}

We note that in applying Theorem~\ref{dt2.1} and Corollary~\ref{dc2.2}, we usually set $p=2$ or $p=1$. We also note that the requirement that~\eqref{d2.5} is valid for all polynomials of total degree $n$ is not crucial for deducing~\eqref{d2.6} but it is for the proof of~\eqref{d2.6'}.

It is already clear from Theorem~\ref{dt2.1} that if~\eqref{d2.1} is satisfied with smallest possible $M>0$, then
\be\label{KK}
\|f\|_{L_\infty(D,\mu)}\le K\|f\|_{L_r(D,\mu)}
\quad\text{for some }r>p
\ee
implies $K\ge M^{p/r}$. In some cases $K$ has to be substantially bigger as we see from the following example.
\begin{example}\label{totik-ex}
$D=T=(-\pi,\pi]$ (the circle), $w(x)=\frac1{2\pi}$ (the Lebesgue measure normalized by $\mu(T)=1$), $p=2$, $r=4$,
\[
S={\rm span}\{e^{i2^kx}:k=1,\dots,n\}.
\]
Clearly $\|f\|_{L_\infty(T)}\le n^{1/2}\|f\|_{L_2(T)}$ for $f\in S$, so $M=n^{1/2}$ (with $p=2$). We choose $f=\sum_{k=1}^ne^{i2^kx}$, $\|f\|_{L_\infty(T)}=f(0)=n$. However,
\begin{align*}
\|f\|_{L_4(T)}^4&=\int_{-\pi}^\pi(f\overline{f})^2d\mu
=\int_{-\pi}^\pi\Biggl(n+\sum_{1\le k< j\le n}\left(e^{i(2^k-2^j)x}+e^{i(2^j-2^k)x}\right)\Biggr)^2d\mu\\
&=n^2+\frac{n(n-1)}2\cdot2=2n^2-n,
\end{align*}
and hence $K$ of~\eqref{KK} satisfies $K\ge n/(2n^2-n)^{1/4}\approx n^{1/2}$. For large $n$ the value of $K$ is much bigger than $M^{1/2}=n^{1/4}$.
\end{example}

The first author seems to remember vaguely that the above example (or a similar one) was given to him  by V.~Totik during a coffee break in a conference in Banff about a decade ago. This example answers negatively the problem posed in~\cite[p.~3244]{Bo-Er} (by choosing $[a,b]=[-\pi,\pi]$ and any fixed $\delta\in(0,\pi)$).


\begin{remark}
In Theorem~\ref{dt2.1} one can replace polynomials of (total) degree $n$ and $D\subset\R^d$ by trigonometric polynomials of (total) degree $n$ and $D\subset T^d$ (the torus) or by spherical harmonic polynomials of degree $n$ and $D\subset S^{d-1}$ (the unit sphere in $\R^d$). In the proof of~\eqref{d2.6} $\P_n$ can be replaced by a subspace of $\P_n$ for which~\eqref{d2.5} is satisfied and for example we may choose to deal with polynomials of degree $n$ orthogonal to polynomials of degree $k$ for some $k<n$.
\end{remark}

\section{Nikol'skii inequality and Christoffel function}\label{sect-chris}

In this section we establish relations between the Nikol'skii inequality and the Christoffel function.


From now on, we consider only real-valued functions. For a bounded domain $D$, $D\subset\R^d$, a finite measure $\mu$ on $D$ and a finite dimensional space $S$ of bounded continuous functions on $D$, the Christoffel function $\lambda(S,D,\mu,\x)$ is given by
\be\label{d3.1}
\lambda(S,D,\mu,\x)^{-1}=C(S,D,\mu,\x)=\sum_{k=1}^N\phi_k(\x)^2,
\quad \x\in D,
\ee
where $\{\phi_k\}_{k=1}^N$ is an orthonormal basis of $S$ on $D$ with respect to $\mu$. Apriori, the Christoffel function depends on the choice of the orthonormal basis of $S$, but we will show below that it does not. This fact and some other basic results of this section are known, but as the proofs are short, we include them for completeness. We also give them in a form and under conditions we need.

Note that we use $N$ to denote the dimension of the space $S$. For spaces of polynomials of degree $n$ we normally have $N\approx n^s$ with appropriate $s$.

Throughout this section we use the function $C(S,D,\mu,\x)$ rather than $\lambda(S,D,\mu,\x)$.

Asymptotics of the Christoffel function for different polynomial spaces have been studied in various papers, see, for example~\cite{To} and references therein. However, known results are usually concerned with the asymptotic behaviour of the Christoffel functions at a fixed point as the degree $n$ approaches infinity. For our problem, a different quantity is of importance, namely, the order (with respect to $n$) of the maximum value of the function $C(S,D,\mu,\x)$ over all $\x\in D$.

The definition~\eqref{d3.1} can be extended to the situation when $S$ is a finite dimensional space of continuous functions on all $\R^d$ (but with the basis $\{\phi_k\}_{k=1}^N$ still orthonormal in $D$), in which case $C(S,D,\mu,\x)$ can be defined for all $\x\in\R^d$, not just for $\x\in D$. In the following few statements and proofs, the corresponding situation is treated in the parentheses. In particular, this is the situation when $S$ consists of polynomials (algebraic, trigonometric or spherical harmonics).

\begin{theorem}\label{dt3.1}
Suppose $D\subset\R^d$ is a bounded domain, $\mu$ is a finite positive measure on $D$, $0<\mu(D)<\infty$, and $S$ is a finite dimensional space of bounded continuous functions on $D$ (or on $\R^d$). Then for $f\in S$ and $\x\in D$ (or $\x\in\R^d$) we have
\be\label{d3.2}
C(S,D,\mu,\x)^{-1/2}=\min_{f\in S,\,|f(\x)|=1}\|f\|_{L_2(D,\mu)}.
\ee
In particular, $C(S,D,\mu,\x)$ does not depend on the choice of the orthonormal basis of $S$ (on $D$).
\end{theorem}
\begin{proof}
Let $\{\phi_k\}_{k=1}^N$ be an orthonormal basis of $S$ on $D$ (or of restrictions of functions from $S$ to $D$). We set $f=\sum_{k=1}^Na_k\phi_k$, $a_k\in\R$, and by the Parseval identity,
\[
\|f\|_{L_2(D,\mu)}=\Bigl(\sum_{k=1}^N a_k^2\Bigr)^{1/2}.
\]
Using the Cauchy-Schwartz inequality,
\[
|f(\x)|=\left|\sum_{k=1}^Na_k\phi_k(\x)\right|
\le\Bigl(\sum_{k=1}^N a_k^2\Bigr)^{1/2}
\Bigl(\sum_{k=1}^N \phi_k(\x)^2\Bigr)^{1/2},
\]
and hence
\[
|f(\x)|\le C(S,D,\mu,\x)^{1/2} \|f\|_{L_2(D,\mu)}.
\]
To show that the minimum in~\eqref{d3.2} is attained, select $a_k=\phi_k(\x)$, $k=1,\dots,N$. As~\eqref{d3.2} is valid for any orthonormal basis, $C(S,D,\mu,\x)$ does not depend on the choice of that basis.
\end{proof}

As a corollary of Theorem~\ref{dt3.1} and Theorem~\ref{dt2.1} we have:
\begin{corollary}\label{dc3.2}
Under the conditions of Theorem~\ref{dt3.1} and setting
\[
\sup_{\x\in D} C(S,D,\mu,\x)=: C(S,D,\mu)
\]
for any $f\in S$, $q\le 2$ and $r>q$, we have
\be\label{3.3}
\|f\|_{L_r(D,\mu)}
\le
C(S,D,\mu)^{(\frac1q-\frac1r)}
\|f\|_{L_q(D,\mu)}.
\ee
\end{corollary}

For $S=\P_n$ a stronger corollary follows, as will be mentioned in Section~\ref{sect-examples}.

The following elementary result (observation) will be useful.
\begin{theorem}\label{dt3.3}
Suppose $D_1\subset D_2\subset\R^d$ are two bounded domains, $S$ is a finite dimensional space of bounded continuous functions on $D_2$ (or on $\R^d$) and $\mu$ a finite measure on $D_2$. Then for $\x\in D_1$ (or for $\x\in\R^d$)
\be\label{d3.4}
C(S,D_1,\mu,\x)\ge C(S,D_2,\mu,\x).
\ee
\end{theorem}
\begin{proof}
For $\x\in D_1$ (or for $\x\in\R^d$) and $f\in S$ we have
\[
C(S,D_2,\mu,\x)^{-1/2}=\min_{f\in S,\,|f(\x)|=1}\|f\|_{L_2(D_2,\mu)}
\ge \min_{f\in S,\,|f(\x)|=1}\|f\|_{L_2(D_1,\mu)}=
C(S,D_1,\mu,\x)^{-1/2}.
\]
\end{proof}

When $\mu$ is the Lebesgue measure ($d\mu=w(\x)d\x$ and $w(\x)=1$), we set the Christoffel function as $\lambda(S,D,\x)=C(S,D,\x)^{-1}$. This situation will be our primary focus for the remainder of the paper.

Let $T\x=\x_0+A\x$ be an affine transformation on $\R^d$. In what follows, it will be understood that $\det T=\det A$. In addition, whenever we refer to an affine transformation $T$ on $\R^d$, we assume that it is non-degenerate, i.e., $\det T\ne0$. The space $S$ of functions on $D$ will be naturally mapped to the space $S^T$ of functions on $TD= \x_0+AD$, the image of $D$ under $T$, by $g(T\x)=f(\x)$, where $g\in S^T$ and $f\in S$. We now track how the Christoffel function will change under the affine transformation of the domain.
\begin{theorem}\label{l3.5}
For any affine transformation $T$ on $\R^d$, domain $D\subset\R^d$, space $S$ of functions on $D$, $\x\in D$ (or $\x\in\R^d$), we have
\be\label{d3.5}
C(S^T,TD,T\x)=C(S,D,\x)|\det T|^{-1}.
\ee
\end{theorem}
\begin{proof}
Clearly,
\[
\|f(T\cdot)\|_{L_2(TD)}=\|f(\cdot)\|_{L_2(D)}\cdot|\det T|^{1/2},
\]
which directly implies~\eqref{d3.5}.
\end{proof}

While very simple, the above theorem will be used frequently. Note that the space of algebraic polynomials is affine-invariant: $\P_n^T=\P_n$.



For convex bodies (convex compact sets with non-empty interior) $D\subset\R^d$, we can show that in order to compute the maximum value (up to a constant factor) of the Christoffel function over $D$, it is sufficient to compute the maximum over the boundary $\partial D$ of $D$. First we establish a geometric lemma and then apply it in our context.
\begin{lemma}\label{boundary_conv_lemma}
Suppose $D$ is a convex body in $\R^d$. Then, for any $\x\in D$ there exists an affine transformation $T$ with $|\det T|\ge \frac1{2^d}$ such that
\[
\x\in T(\partial D)\quad\text{and}\quad T(D)\subset D.
\]
\end{lemma}
\begin{proof}
For any $\x\in D$, consider any segment $[\y,\z]$ containing $\x$ with $\y,\z\in\partial D$, $\y\ne\z$. Without loss of generality, we can assume that $\lambda:=\frac{|\x-\y|}{|\z-\y|}\ge \frac12$. Take $T$ to be the homothety with coefficient $\lambda$ and center at $\y$, i.e., $T(\cdot)=\y+\lambda(\cdot-\y)=(1-\lambda)\y+\lambda(\cdot)$. Clearly, $|\det T|=\lambda^d\ge\frac1{2^d}$. Also, $\x=T(\z)\in T(\partial D)$. Finally, as $0<\lambda<1$ and $\y\in D$, by convexity of $D$, the image $T(D)$ is a subset of $D$.
\end{proof}
\begin{theorem}\label{boundary_const}
For any convex body $D\subset\R^d$
\[
\max_{\x\in D} C(\P_n, D, \x)\le 2^d \max_{\x\in\partial D} C(\P_n,D,\x).
\]
\end{theorem}
\begin{proof}
Assume that the maximum on the left hand side is attained at a point $\tx\in D$. Apply Lemma~\ref{boundary_conv_lemma} for $\tx$, let $T$ be the resulting affine transformation. By Theorem~\ref{dt3.3},
\[
\max_{\x\in D} C(\P_n, D, \x)=C(\P_n, D, \tx)\le C(\P_n,T(D),\tx) \le \max_{\x\in T(\partial D)} C(\P_n,T(D),\x).
\]
Theorem~\ref{l3.5} implies
\[
\max_{\x\in T(\partial D)} C(\P_n,T(D),\x)=|\det T|^{-1} \max_{\x\in\partial D} C(\P_n,D,\x),
\]
and 
the proof is complete.
\end{proof}

Further strengthening of Theorem~\ref{boundary_const} is possible allowing us to prove a similar inequality with $\partial D$ replaced by a smaller set (even one-point set for some domains), see Section~\ref{sect-sharp}.

\begin{remark}
For some $D$ we know that
\[
\max_{\x\in D} C(\P_n,D,\x)=\max_{\x\in\partial D} C(P_n,D,\x).
\]
For the unit Euclidean ball (or ellipsoid) the maximum is achieved at all points of the boundary (see Theorem~\ref{ball-nik}). For the cube (or its affine transformation) the maximum is achieved at all vertices (see Theorem~\ref{cube_general}). Moreover, the proof of Theorem~\ref{dt3.4} implies $\max_{\x\in\partial D}C(\P_n,D,\x)\ge c_1(R) n^{d+1}$ for any domain $D$, $D\subset \x_0+RB_2^d$ (see~\eqref{ball-def}), and~\eqref{dC} of Theorem~\ref{ball-nik} implies for any point $\x_1$ such that ${\rm dist}\,(\x_1,\partial D)\ge r$ implies $C(\P_n,D,\x_1)\le c_2(r) n^d$. Therefore, for $n$ large enough $\max C(\P_n,D,\x)$ has to be attained at a point close to the boundary. This gives credence to our conjecture below that $\max_{\x\in D} C(\P_n,D,\x)=\max_{\x\in\partial D} C(P_n,D,\x)$ for any convex domain $D\subset \R^d$.
\end{remark}

An extreme point of a convex set $D\subset\R^d$ is a point in $D$ which does not lie in any open line segment joining two points of $D$. Alternatively, $\y\in D$ is extreme if $D\setminus\{\y\}$ is convex.

\begin{conjecture}
For a convex compact set $D\subset\R^d$, $\max_{\x\in D} C(\P_N,D,\x)$ is achieved at an extreme point of $D$.
\end{conjecture}

\section{Christoffel function on the ball and the cube, and applications}\label{sect-christ-ball-cube}

Let
\begin{equation}\label{ball-def}
B_2^d=\{\x=(x_1,\dots,x_d)\in\R^d:\|\x\|_2=(x_1^2+\dots+x_d^2)^{1/2}\le1\}
\end{equation}
be the unit Euclidean ball in $\R^d$. (Later we will make use of the notation $B_\alpha^d$ for the more general $l_\alpha$ unit balls.) The ball of radius $r$ centered at $\x_0$ is $\x_0+rB_2^d$.
We also define
\[
C(S,D):= \sup_{\x\in D} C(S,D,\x).
\]
\begin{theorem}\label{ball-nik}
We have
\begin{align}
&C(\P_n,B_2^d)\approx n^{d+1}\text{ and }C(\P_n,B_2^d,\x)=C(\P_n,B_2^d)\text{ for any }\x:\ \|\x\|_2=1;\label{dA}\\
&C(\P_n,B_2^d,\x)\text{ is an increasing function of }\|\x\|_2\text{ when }\frac12\le\|\x\|_2^2\le 1;\label{dB}\\
&C(\P_n,B_2^d,\x)\approx n^d\text{ for }\|\x\|_2<a<1\text{ with constants of equivalence depending on }a.\label{dC}
\end{align}
\end{theorem}
\begin{proof}
With $\lambda=\frac12$, \cite[Theorem~2.3]{Di14} implies~\eqref{dA}, and~\cite[Lemma~2.1 and~(2.3)]{Di14} yield~\eqref{dB}. The lower bound $C(\P_n,B_2^d,\x)\ge c_1n^d$ for $\|\x\|_2\le a<1$ is given in the proof of~\cite[Theorem~2.3]{Di14}.
For the proof in the other direction, it is sufficient to show $C(\P_n,B_2^d,\zero)\le c_2n^d$, and use Theorems~\ref{dt3.3} and~\ref{l3.5} as $\x+(1-a)B_2^d\subset B_2^d$ for $\|\x\|_2\le a$.

Using the orthonormal system $P_{n,j,l}^{\lambda,d}$ given in~\cite[(2.1)]{Di14} for $\lambda=\frac12$, we note that
\[
I:= \sum_{k=0}^n\sum_{j=0}^{k-2j}\sum_{l=1}^{d_{k-2j}}\left(P_{k,j,l}^{(\frac12,d)}(\zero)\right)^2
=\sum_{m=0}^{\lfloor \frac n2\rfloor} \left(P_{2m,m,1}^{(\frac12,d)}(\zero)\right)^2
\]
as all the other terms contain $\|\x\|_2$ to a positive power and hence are equal to zero at $\x=\zero$. We now have
\[
I=2^{(d+3)/2}\sum_{m=0}^{\lfloor \frac n2\rfloor}\left(P_m^{(0,(d-2)/2)}(-1)\right)^2 \approx \sum_{m=0}^{\lfloor \frac n2\rfloor} m^{d-1}\approx n^d
\]
as $P_m^{(0,(d-2)/2)}(-1)$, the normalized Jacobi polynomials, satisfy $P_m^{(0,(d-2)/2)}(-1)\approx m^{(d-1)/2}$ which is shown in~\cite[p.~168]{Sz}. We note that in~\cite{Sz} the value at $-1$ is given in the proof, more precisely, $P_m^{(0,(d-2)/2)}(-1)=\binom{m+(d-2)/2}{m}$ for the orthogonal \emph{nonnormalized} Jacobi polynomials, and the normalization contributes an additional factor equivalent to $m^{1/2}$, see~\cite[(4.3.3), p.~68]{Sz}.
\end{proof}

\begin{theorem}\label{dt3.4}
Suppose $D\subset\R^d$ is a compact set. Then
\be\label{d3.6}
C(\P_n,D)= \max_{\x\in D} C(\P_n,D,\x)\ge c n^{d+1}
\ee
and $c$ depends only on the diameter of $D$.
\end{theorem}
\begin{proof}
By Theorem~\ref{ball-nik},
\[
C(\P_n,B_2^d,\x)=c_1n^{d+1}
\]
for any $\x$ on the sphere $\|\x\|_2=1$, and hence using Theorem~\ref{l3.5},
\[
C(\P_n,\x_0+rB_2^d,\x)=cn^{d+1}
\]
for any $\x$ satisfying $\|\x-\x_0\|_2=r$. Clearly, $D\subset \x_0+rB_2^d$ with $r$ half the diameter of $D$. There is $\x\in{D}$ such that $\|\x-\x_0\|_2=r$ and hence for that $\x$
\[
C(\P_n,{D},\x)\ge cn^{d+1}.
\]
As polynomials are continuous and $D$ is compact,
\[
C(\P_n,D)=\sup_{\x\in D}C(\P_n,D,\x)=\max_{\x\in D}C(\P_n,D,\x),
\]
which concludes the proof.
\end{proof}


\begin{theorem}\label{dt3.5}
Suppose $D\subset\R^d$, $D$ is a compact set, and $D\supset \x_1+r_1B_2^d$. Then
\[
C(\P_n,D,\x)\le c_2 n^{d+1}
\quad\text{for }
\x\in \x_1+r_1B_2^d.
\]
If in addition ${D}=\bigcup_\lambda (\x_\lambda+r_\lambda B_2^d)$ and $\inf_\lambda r_\lambda=r>0$, then
\[
C(\P_n,D)\approx n^{d+1}.
\]
\end{theorem}
\begin{proof}
As Theorem~\ref{l3.5} and~\cite{Di14} imply $C(\P_n,\x_1+r_1B_2^d,\x)\le c n^{d+1}$ for $\x\in \x_1+r_1B_2^d$, \eqref{d3.4} implies the inequality. Theorem~\ref{dt3.4} and the first inequality imply the equivalence where the constants of the equivalence depend on the diameter of $D$ and on $r$.
\end{proof}

\begin{theorem}\label{cube_general}
Suppose $Q=[-1,1]^d=B_\infty^d$ is the cube in $\R^d$. Then $C(\P_n,Q)\approx n^{2d}$ and $C(\P_n,Q,(\pm1,\dots,\pm1))\ge c(d)n^{2d}$.
\end{theorem}
\begin{proof}
Using~\cite[Theorem~6.6, pp.~119--120]{DiTi} with $\alpha_i=\beta_i=0$, we have $C(\P_n,Q)\le c_1 n^{2d}$. Using~\cite[(7.32.2), p.~168]{Sz}, the one-dimensional orthonormal polynomials with weight $w(x)=1$ ($\alpha_i=\beta_i=0$) each satisfies $P_k(\pm1)\ge c_1 k^{1/2}$ as it is shown that that the maximum occurs at both $-1$ and $1$ (whenever $\alpha=\beta\ge-\frac12$). Therefore,
\begin{align*}
C(\P_n,B_\infty^d,(\pm1,\dots,\pm1)) &= \sum_{l=0}^n \prod_{\sum_{i=1}^dk_i=l}(P_{k_i}(\pm1))^2\\
&\ge \prod_{i=1}^d\left(\sum_{k_i=0}^{\lfloor \frac nd\rfloor}(P_{k_i}(\pm1))^2\right) \ge c^d \prod_{i=1}^d \sum_{k_i=0}^{\lfloor \frac nd\rfloor} k_i \\
&\ge c_1^d \prod_{i=1}^d c_2 \left(\frac nd\right)^2\ge c(d) n^{2d}.
\end{align*}
Hence, $C(\P_n,B_\infty^d)\approx n^{2d}$.
\end{proof}


\begin{theorem}\label{dt3.7}
Suppose for an affine transformation $T$ on $\R^d$ and the cube $Q=[-1,1]^d=B_\infty^d$ a compact set $D$ satisfies $D\subset TQ$ and suppose further that $\x=T(1,\dots,1)\in D$. Then
\[
C(\P_n,D,\x)\ge c_1 n^{2d} \quad\text{and}\quad C(\P_n,D)\ge c_1 n^{2d},
\]
where $c_1$ depends on $\det T$ and $d$.
\end{theorem}
\begin{proof}
Using Theorem~\ref{dt3.3}, Theorem~\ref{l3.5}, and Theorem~\ref{cube_general}, we conclude that
\[
C(\P_n,D) \ge C(\P_n,D,\x) \ge C(\P_n, TQ, \x)=|\det T|^{-1} C(\P_n,Q,(1,\dots,1))\ge c_1 n^{2d}.
\]
\end{proof}

\begin{theorem}\label{cube-in}
For any convex body $D\subset\R^d$ we have
\[
C(\P_n,D)\le c n^{2d},
\]
where $c>0$ depends only on $D$ and does not depend on $n$.
\end{theorem}
\begin{proof}
Since $D$ has non-empty interior, there is a Euclidean ball $B\subset D$. For any $\x\in D$, the convex hull of $\{\x\}$ and $B$ contains an affine image $TQ$ of the cube $Q=[-1,1]^d=B_\infty^d$ with $T(1,\dots,1)=\x$ and $|\det T|>c_1$ for some $c_1>0$ independent of $\x$. Then by Theorem~\ref{dt3.3}, Theorem~\ref{l3.5}, and Theorem~\ref{cube_general},
\[
C(\P_n,D,\x) \le C(\P_n, TQ, \x)=|\det T|^{-1} C(\P_n,Q,(1,\dots,1))\le c n^{2d},
\]
and the proof is complete as $\x\in D$ was chosen arbitrarily.
\end{proof}



Some simple examples of applications of the above theorems in the context of Nikol'skii inequalities are given in Section~\ref{sect-examples}.

\section{Upper estimate of $C(\P_n,D)$}\label{sect-upper}

The quite elementary comparison result Theorem~\ref{dt3.3} used in conjunction with Theorem~\ref{l3.5} will be extremely useful for lower and upper estimates of $C(\P_n,D)$. A simple illustration of this idea is Theorem~\ref{dt3.5}, where a Euclidian ball is inscribed into the domain. This may fail to work near the boundary of a domain with large curvature. To remedy this, we will employ two more ideas: affine transformations to ``squeeze'' the ball, and an ``extension'' that allows one to step away slightly from the boundary of the domain. We begin with an estimate of $C(\P_n, B_2^d, \v_n^d)$, where the point $\v_n^d:=(1+n^{-2}/3,0,\dots,0)\in\R^d$ is outside of the ball $B_2^d$.
\begin{lemma}\label{refball}
We have
\[
C(\P_n, B_2^d, \v_n^d)\le c n^{d+1},
\]
where $c$ depends only on $d$.
\end{lemma}
\begin{proof}
Let $P_n\in\P_n$ be such that $P_n(\v_n^d)=1$ and $C(\P_n, B_2^d, \v_n^d)=\|P_n\|_{L_2(B_2^d)}^{-2}$. We claim that $\|P_n\|_{L_\infty(B_2^d)}\ge\frac12$. Let $\y$ be the point on the segment joining $\x:=(1,0,\dots,0)$ and $\v_n^d$ where $|P_n|$ attains maximum value. As $P_n(\v_n^d)=1$, we have $|P_n(\y)|=:M\ge1$. If we assume $\|P_n\|_{L_\infty(B_2^d)}<\frac12$, then the largest value of $|P_n|$ on the segment $I$ joining $-\x$ and $\y$ is $M$, and it is attained at $\y$. We apply Markov's inequality on $I$ for $P_n$ and the mean value theorem to obtain (note that the length of $I$ is at most $2+n^{-2}/3\le\frac73$)
\[
|P_n(\x)-P_n(\y)|\le \frac{7}6 n^2 |\x-\y| M\le \frac7{18}M<\frac12M.
\]
Hence, $|P_n(\x)|\ge M/2\ge1/2$, which leads to a contradiction.

With $Q_n(\cdot):=P_n(\cdot)/\|P_n\|_{L_\infty(B_2^d)}$, and using Theorem~\ref{ball-nik}, we conclude
\[
C(\P_n, B_2^d, \v_n^d)=\|P_n\|_{L_2(B_2^d)}^{-2}\le 4 \|Q_n\|_{L_2(B_2^d)}^{-2}\le 4 C(\P_n,B_2^d) \le c n^{d+1}.
\]
\end{proof}

The following result is the core of our ``extension'' technique.
\begin{theorem}\label{main-ell}
Suppose $D\subset\R^d$ is a compact set, $T$ is an affine transformation of $\R^d$ such that $TB_2^d\subset D$. Then
\[
C(\P_n,D,T\v_n^d)\le c |\det T|^{-1}n^{d+1}
\]
where $c$ depends only on $d$.
\end{theorem}
\begin{proof}
The result follows immediately from Theorem~\ref{dt3.3}, Lemma~\ref{refball} and Theorem~\ref{l3.5}.
\end{proof}

According to the above theorem, the Christoffel function $C(\P_n,D,\x)$ can be bounded from above if the appropriate affine image of $B_2^d$ can be inscribed into the domain, while mapping $\v_n^d$ into $\x$. This is applicable for any fixed $n$. Our next result provides a bound for all $n$ if a certain ``cone''-type set (as in~\eqref{cone-in}) can be inscribed into the domain. Note that for $s=1$ \eqref{cone-in} describes a right circular cone, while for $s=1/2$ we obtain almost a spherical cap.

\begin{theorem}\label{cone-set-ell}
Suppose $D\subset\R^d$ is a compact set, $\x\in D$, the vector $\u\in\R^d$ has unit length, $\beta_1$ and $\beta_2$ are positive constants, $s\in[\frac12,1]$, and furthermore
\begin{equation}\label{cone-in}
\{\x+\delta\u+\lambda \delta^{s} \v: \delta\in[0,\beta_1], \lambda\in[0,\beta_2], \v\in\R^d, |\v|=1, \u\perp\v \}\subset D.
\end{equation}
Then
\[
C(\P_n,D,\x)\le c n^{2+2s(d-1)}, \quad n\ge1,
\]
where $c$ depends only on $\beta_1$, $\beta_2$, and $d$.
\end{theorem}
\begin{proof}
We need to construct an affine transformation $T$ with not too small $|\det T|$ such that $T\v_n^d=\x$ and $TB_2^d\subset D$. The idea is that $T$ maps the $x_1$ axis to the line through $\x$ parallel to $\u$. More precisely, take $T(\cdot)=\x+A(\cdot-\v_n^d)$, where the matrix $A$ is decomposed as $A=A_2A_1$. Let $A_2$ be any rotation of $\R^d$ mapping $(-1,0,\dots,0)$ to $\u$, and let $A_1(x_1,\dots,x_d)=(\frac{\beta_1}3 x_1,\mu x_2,\dots,\mu x_d)$, where $\mu=\frac{{\beta_2}}{\sqrt{6}}n^{-2s+1}$. We have $T\v_n^d=\x$ and
\[
T(-1,0,\dots,0)=\x+(2+n^{-2}/3)\frac{\beta_1}3 \u,
\]
where clearly $(2+n^{-2}/3)\frac13<1$.

To check that $TB_2^d\subset D$, we will use the condition~\eqref{cone-in}.
Note that $\sqrt{1-(1-t)^2}\le \sqrt{2t}$, $t\in[0,2]$. Therefore, any point of $B_2^d$ with its first coordinate equal to $1-t$ is located at most $\sqrt{2t}$ away from $(1-t,0,\dots,0)$ in a direction orthogonal to $(1,0,\dots,0)$. This fact and the structure of $T$ imply that it is sufficient to prove
\[
\mu\sqrt{2t}\le {\beta_2} (t+n^{-2}/3)^s, \quad t>0.
\]
Indeed, since $s\in[1/2,1]$, we have
\begin{align*}
\mu\sqrt{2t} &= \frac{{\beta_2}}{\sqrt{3}} n^{-2s+1}t^{1/2} = \frac{{\beta_2}}{\sqrt{3}} 3^{s-1/2} (n^{-2}/3)^{s-1/2}t^{1/2} \\
&\le {\beta_2} (\max\{n^{-2}/3,t\})^{(s-1/2)+1/2}\le {\beta_2} (t+n^{-2}/3)^s.
\end{align*}
Finally, $|\det T|=\frac{\beta_1}3\mu^{d-1}= c_1 n^{(-2s+1)(d-1)}$, and Theorem~\ref{main-ell} yields
\[
C(\P_n,D,T\v_n^d)\le c_2 |\det T|^{-1}n^{d+1}\le c n^{(2s-1)(d-1)+d+1}=c n^{2+2s(d-1)}
\]
as required.
\end{proof}

\begin{remark}
If $D$ is convex and we are interested in estimating $C(\P_n,D)$ from above, then Theorem~\ref{boundary_const} (or Theorem~\ref{sharp_th}) allows one to consider only the case when the point $T\v^d_n$ in Theorem~\ref{main-ell} or the point $\x$ in Theorem~\ref{cone-set-ell} belongs to the boundary (or to a sharp subset, see Definition~\ref{sharp_def}) of $D$.
\end{remark}

\section{Lower estimate of $C(\P_n,D)$}\label{sect-lower}

To obtain lower estimates of $C(\P_n,D)$ we construct examples of algebraic polynomials that have uniform norm of constant order and a small $L_2$ norm. The main ingredient of such examples will be good univariate polynomials constructed in Lemma~\ref{kernel_pol}. We thank Igor Shevchuk for suggesting the reference~\cite{KLS}, which contains useful technical details that allowed us to simplify the proof of the lemma. We also note that this lemma can be obtained from the somewhat stronger result~\cite[5.6.25b, p.~250]{Tr-Be} proved in~\cite{Tr}. However, the article~\cite{Tr} is hard to access, so we include our relatively short proof here.


\begin{lemma}\label{kernel_pol}
Denote $\rho_n(x):=\frac1{n^2}+\frac1n\sqrt{1-x^2}$. For any $n,m\ge1$ and $y\in[-1,1]$,  there exists a polynomial $P_n=P_{n,m,y}$ of degree $\le n$ such that
\begin{equation}\label{kernel_low}
P_n(y)=1,
\end{equation}
and
\begin{equation}\label{kernel_up}
|P_n(x)|\le c_1 \left(\frac{\rho_n(y)}{\rho_n(y)+|x-y|}\right)^m, \quad x\in[-1,1],
\end{equation}
where $c_1>0$ depends only on $m$.
\end{lemma}
\begin{proof}
 For $k:=\lfloor \frac nm \rfloor+1$, $k\ge 1$ and $(k-1)m\le n<km$. Let $T_k(x):=\cos(k\arccos(x))$ be the Chebyshev polynomial of degree $k$, and let $x_j=\cos(j\pi/k)$, $0\le j\le k$, be the corresponding Chebyshev partition of $[-1,1]$. For any $y\in[-1,1]$, we can choose $j$, $0\le j\le k$, so that $y\in[x_j,x_{j-1}]$. We will use the polynomial $t_j$ from~\cite[Lemma~A.1, p.~1246]{KLS} defined as
\[
t_j(x):=\frac{T_k(x)}{x-\tilde x_j}(x_{j-1}-x_j),\ x\ne\tilde x_j, \quad\text{and}\quad t_j(\tilde x_j):=T_k'(\tilde x_j)(x_{j-1}-x_j),
\]
where $\tilde x_j=\cos(j-\frac12)\frac\pi k$ is the zero of $T_k$ lying in $[x_j,x_{j-1}]$. Then by~\cite[Lemma~A.1]{KLS},
\begin{equation}\label{a1}
\frac 43<|t_j(x)|<4, \quad x\in[x_j,x_{j-1}].
\end{equation}
We define now $P_n=P_{n,m,y}$:
\[
P_n(x):=\left(\frac{t_j(x)}{t_j(y)}\right)^m, \quad x\in[-1,1].
\]
Clearly, $P_n$ is an algebraic polynomial of degree $\le (k-1)m\le n$ satisfying~\eqref{kernel_low}.

Straightforward arguments imply
\[
\rho_n(x)\approx \rho_k(x), \quad x\in[-1,1],
\]
with constants of equivalence depending only on $m$. Therefore, to prove~\eqref{kernel_up}, it is sufficient to show
\begin{equation}\label{tj_bound}
|t_j(x)|\le c \frac{\rho_k(y)}{\rho_k(y)+|x-y|}, \quad x\in[-1,1],
\end{equation}
with an absolute constant $c$. For~\eqref{tj_bound}, we use the fact that following~\eqref{a1}, $t_j(y)\ge \frac43$.

We will need the following useful estimates (see~\cite[(5.6)]{KLS}):
\begin{equation}\label{kls-est}
\rho_k(y)\le x_{j-1}-x_j\le \frac{\pi^2}2\rho_k(y),
\end{equation}
which are, in fact, valid for any choice of $y\in[x_j,x_{j-1}]$.

If $x\in[x_j,x_{j-1}]$, then by~\eqref{kls-est} we have $\frac{\pi^2}2\rho_k(y)\ge x_{j-1}-x_j\ge|x-y|$, and by~\eqref{a1}
\[
|t_j(x)|\le 4= \tilde c \frac{1}{1+\frac{\pi^2}2}\le \tilde c\frac{\rho_k(y)}{\rho_k(y)+|x-y|},
\]
and, therefore, \eqref{tj_bound} holds for such $x$. If $x\not\in[x_j,x_{j-1}]$, then by~\cite[(5.5)]{KLS} we have $|x-\tilde x_j|\ge \frac14(x_{j-1}-x_j)$, and hence, by~\eqref{kls-est}, $8|x-\tilde x_j|\ge 2\rho_k(y)$. Further, using~\eqref{kls-est} again, we get $|x-\tilde x_j|\ge |x-y|-|y-\tilde x_j|\ge |x-y|-\rho_k(y)$. So, $9|x-\tilde x_j|\ge \rho_k(y)+|x-y|$, and by $|T_k(x)|\le1$ and~\eqref{kls-est}, we conclude
\[
|t_j(x)|\le \frac{x_{j-1}-x_j}{|x-\tilde x_j|}\le  \frac{9\pi^2}2 \cdot \frac{\rho_k(y)}{\rho_k(y)+|x-y|},
\]
which completes the proof of~\eqref{tj_bound} for all $x$.
\end{proof}

We will present two constructions.

First, we can obtain a lower estimate of $C(\P_n,D)$ in terms of parallel section functions of $D$. For $\x,\y\in\R^d$, we denote by $\x\cdot\y$ the dot product of $\x$ and $\y$ in $\R^d$. For a unit vector $\bxi\in\R^d$ and a non-empty compact set $D\subset\R^d$, we define the parallel section function as
\[
A_{D,\bxi}(t):=\Vol_{d-1} (D\cap\{\x\in\R^d:\x\cdot\bxi=t+h\}),
\]
where $h\in\R$ is the smallest $h$ such that $D\cap\{\x\in\R^d:\x\cdot\bxi=t+h\}\ne\emptyset$, and $\Vol_{d-1}$ is the $(d-1)$-dimensional Lebesgue measure.
\begin{theorem}\label{parallel}
For any non-empty compact set $D\subset\R^d$, unit vector $\bxi\in\R^d$, and $n,m\ge1$, we have
\begin{equation}\label{two_integ}
C(\P_n,D)\ge c \left(\int_0^{n^{-2}} A_{D,\bxi}(t) \,dt + n^{-2m} \int_{n^{-2}}^\infty A_{D,\bxi}(t) t^{-m}\,dt\right)^{-1},
\end{equation}
where $c=c_{m,D}>0$ depends only on $m$ and $D$ (in fact, only on the diameter of $D$), and does not depend on $n$.

In particular, if $A_{D,\bxi}(t)\le M t^\lambda$ for some $M>0$, $\lambda<m-1$ and all $t>0$, then
\begin{equation}\label{pow_beh}
C(\P_n,D)\ge c n^{2(1+\lambda)}
\end{equation}
where $c$ may, in addition, depend on $M$ and $\lambda$.
\end{theorem}
\begin{proof}
Let $h\in\R$ be from the definition of $A_{D,\bxi}(t)$. Since $D$ is bounded, we can choose $b>0$ such that $0\le \x\cdot\bxi-h\le 2b$ for any $\x\in D$. Let $P_n=P_{n,m,1}$ be the polynomial from Lemma~\ref{kernel_pol} with $y=1$ of degree $\le n$. Define $Q_n(\x):=P_n(1-\frac{\x\cdot\bxi-h}b)$. By~\eqref{kernel_low}, $\|Q_n\|_{C(D)}\ge P_n(1)=1$, so using~\eqref{kernel_up}, we proceed as follows:
\begin{align*}
C(\P_n,D)^{-1} &\le \|Q_n\|_{L_2(D)}^{2}=\|Q_n^2\|_{L_1(D)}=\int_0^{2b} P_n^2(1-tb^{-1})A_{D,\bxi}(t) \,dt\\
 & \le c \int_{0}^\infty \left(\frac{n^{-2}}{n^{-2}+t}\right)^{2m}A_{D,\bxi}(t) \,dt
 \le c \int_{0}^\infty \left(\frac{n^{-2}}{n^{-2}+t}\right)^{m}A_{D,\bxi}(t) \,dt,
\end{align*}
which implies~\eqref{two_integ}. The inequality~\eqref{pow_beh} follows as an immediate corollary.
\end{proof}

The use of the parallel section function will be illustrated in the proof of Theorem~\ref{balls_p_le2}.

The second construction uses the tensor product of polynomials from Lemma~\ref{kernel_pol}.
\begin{theorem}\label{main-lower}
Let $D\subset\R^d$ be a compact set, $\y=(y_1,\dots,y_d)\in[-1,1]^d$, $T$ be an affine transformation of $\R^d$ such that $D\subset T([-1,1]^d)$ and $T\y\in D$. Then
\[
C(\P_n,D,T\y)\ge c |\det T|^{-1} \rho_n^{-1}(y_1)\rho_n^{-1}(y_2)\dots \rho_n^{-1}(y_d)
\]
where $c>0$ depends only on $d$.
\end{theorem}
\begin{proof}
 By Theorem~\ref{dt3.3} and Theorem~\ref{l3.5},
\[
C(\P_n,D,T\y)\ge C(\P_n,T([-1,1]^d),T\y)=|\det T|^{-1} C(\P_n,[-1,1]^d,\y).
\]
Let $\tilde n=\lfloor n/d\rfloor$. Using Lemma~\ref{kernel_pol} with $m=1$, we define
\[
Q_n(x_1,\dots,x_d)=P_{\tilde n,1,y_1}(x_1) P_{\tilde n,1,y_2}(x_2)\dots P_{\tilde n,1,y_d}(x_d).
\]
Then $Q_n$ is a polynomial of total degree $\le n$ and $Q_n(\y)=1$. Using~\eqref{kernel_up}, we have
\[
\|P_{\tilde n,1,y}\|^2_{L_2([-1,1])}\le c \rho_{\tilde n}(y).
\]
Therefore,
\[
C(\P_n,[-1,1]^d,\y)^{-1}\le \|Q_n\|^2_{L_2([-1,1]^d)}\le c' \rho_{\tilde n}(y_1)\rho_{\tilde n}(y_2)\dots \rho_{\tilde n}(y_d)
\le \tilde c \rho_n(y_1)\rho_n(y_2)\dots \rho_n(y_d),
\]
where in the last step we used $\rho_{\tilde n}(t)\le c(d) \rho_n(t)$ for $t\in[-1,1]$.
\end{proof}

A simple example can be the estimate $C(\P_n,B^d_2)\ge cn^{d+1}$. Take $T$ to be the identity, $\y=(1,0,\dots,0)$, the inequality follows by $\rho_n(0)\le 2 n^{-1}$ and $\rho_n(1)=n^{-2}$. Another immediate example is the lower bound of Theorem~\ref{cube_general}. More examples will be given later.


\section{Applications, some simple examples}\label{sect-examples}

In this section we use the results of Sections~\ref{sect-gen}, \ref{sect-chris}, and~\ref{sect-christ-ball-cube} to find the power $\sigma=\sigma(D)$ for many domains $D$ in the Nikol'skii inequality given by
\[
\|\phi\|_{L_q(D)}\le c n^{\sigma(\frac1p-\frac1q)}\|\phi\|_{L_p(D)},\quad\text{for }\phi\in\P_n\quad\text{and}\quad 0<p\le q\le\infty.
\]
As is common in this type of question, we seek to determine $\sigma$ in
\begin{equation}\label{e1}
\sup_{\phi\in\P_n} \|\phi\|_{L_\infty(D)}/\|\phi\|_{L_2(D)}\approx n^{\sigma/2}
\end{equation}
with the constants of the equivalence~\eqref{e1} depending on $D$ but not on $n$. Recall that by~\eqref{d2.6'} of Corollary~\ref{dc2.2}, the Nikol'skii inequality for $(p,q)=(2,\infty)$ with some $\sigma$ (determined in~\eqref{e1}) implies the Nikol'skii inequality for all $(p,q)$, $0<p\le q\le\infty$, with the same $\sigma$.

The equivalence~\eqref{e1} implies that the power $\sigma$ of~\eqref{e1} is the optimal $\sigma$ of the Nikol'skii inequality as it is achieved for $p=2$ and $q=\infty$. In fact, \eqref{e1} implies that it is achieved for $q=\infty$ and any $p$ as supposing that for some $\sigma_1<\sigma$
\[
\|\phi\|_{L_\infty(D)}\le c n^{\sigma_1/p}\|\phi\|_{L_p(D)}
\]
contradicts~\eqref{e1} when we use Corollary~\ref{dc2.2}.


We will also try to find $\x_0\in D$ such that
\begin{equation}\label{e2}
\sup_{\phi\in\P_n}|\phi(\x_0)|/\|\phi\|_{L_2(D)}\approx \sup_{\phi\in\P_n} \|\phi\|_{L_\infty(D)}/\|\phi\|_{L_2(D)}\approx n^{\sigma/2}
\end{equation}
where $\x_0$ does not depend on $n$. Of course, there is more than one possibility for $\x_0$ even when $\|\phi\|_{L_\infty(D)}$ of~\eqref{e1} is equal to $|\phi(\x_0)|$ of~\eqref{e2}. In situations discussed here $\x_0$ does not depend on $n$. We also find for various $\y\in D$
\begin{equation}\label{e3}
\sup_{\phi\in\P_n}|\phi(\y)|/\|\phi\|_{L_2(D)}\approx n^{\sigma_1/2}
\end{equation}
with some $\sigma_1<\sigma$, and examples of such will be useful. The main tools in this section will be the comparison with domains $D_1$ and $D_2$ where $D_1\subset D\subset D_2$ (Theorem~\ref{dt3.3}), estimates for the unit ball (Theorem~\ref{ball-nik}), the cube (Theorem~\ref{cube_general}), the effect of affine transformation of various domains (Theorem~\ref{l3.5}) and the relation with the Christoffel function (Theorem~\ref{dt3.1}). In particular, note that $C(\P_n,D)\approx n^\sigma$ is equivalent to~\eqref{e1}.

We first observe the situation for an interior point $\y$ of $D$.
\begin{observation}\label{eo1}
For domain $D\subset\R^d$ with diameter $R$ and point $\y$ such that $\y+rB_2^d\subset D$ we have
\[
\sup_{\phi\in\P_n}|\phi(\y)|/\|\phi\|_{L_2(D)}\approx n^{d/2}
\]
with the constants of equivalence depending only on $r$, $R$ and $d$.
\end{observation}
\begin{proof}
We use $D_1=\y+rB^d_2$, $D_2=\y+RB_2^d$, $D_1\subset D\subset D_2$, \eqref{dC} of Theorem~\ref{ball-nik}, and Theorem~\ref{l3.5} to obtain $\sup_{\phi\in\P_n}|\phi(\y)|/\|\phi\|_{L_2(D_i)}\approx n^{d/2}$ for $i=1,2$,
and this in turn implies our equivalence.
\end{proof}

We now give a few examples using estimates on affine transformation of the cube.
\begin{example}\label{ee2}
For $D\subset\R^2$ given by
\[
D\text{ is \raisebox{-1.7mm}{\Large \DavidStarSolid}}
\qquad\text{or}\qquad
D\text{ is \raisebox{-1.5mm}{\Large \FiveStar}}
\]
(star of David or five-pointed star), we have
\[
\sup_{\phi\in\P_n}\|\phi\|_{L_\infty(D)}/\|\phi\|_{L_2(D)}\approx n^{2},
\]
i.e. $\sigma(D)=4$.
\end{example}
\begin{proof}
We use for any extreme point $\x\in D$ (a point that is not an interior point of a segment with two different endpoints in $D$) two affine maps of the square (parallelograms), $D_1$ and $D_2$, each containing $\x$ as its extreme point (vertex) with $D_1\subset D\subset D_2$ and observe that on the square $S$ (see~\cite[Th.~6.6]{DiTi} and Theorem~\ref{cube_general}) $\sup_{\phi\in\P_n}|\phi(\x)|/\|\phi\|_{L_2(S)}=\sup_{\phi\in\P_n}\|\phi\|_{L_\infty(S)}/\|\phi\|_{L_2(S)}\approx n^{2}$ for any extreme point (vertex) of $S$.
\end{proof}


\begin{remark}
The above example remains valid for $D$ any polygon in $\R^2$ (not necessarily convex) with exactly the same proof.
\end{remark}

For the next example we recall that a (convex) polytope $D$ in $\R^d$ is a convex hull of finitely many points such that $D$ has a non-empty interior. $S_k$ is a $k$-face ($0<k<d$) of a polytope $D$ if $S_k=H\cap D$ for a supporting hyperplane $H$ of $D$ and $S_k$ has dimension $k$.
\begin{example}\label{ee3}
Suppose $D\subset\R^d$ is a polytope and $\x$ is any extreme point of $D$. Then
\be\label{e4}
\sup_{\phi\in\P_n}|\phi(\x)|/\|\phi\|_{L_2(D)}\approx
\sup_{\phi\in\P_n}\|\phi\|_{L_\infty(D)}/\|\phi\|_{L_2(D)}\approx n^{d},
\ee
and the constants of the equivalences depend only on $d$ and the diameter and shape of $D$.

Suppose in addition that $S_k$ is a $k$-face of $D$ and $\y\in S_k$ satisfies $(\y+rB_2^d)\cap H\subset S_k$ (with $H$ of the definition of $S_k$), in other words, $\y$ is in the relative interior of $S_k$. Then
\be\label{e5}
\sup_{\phi\in\P_n}|\phi(\y)|/\|\phi\|_{L_2(D)}\approx n^{d-k/2}
\ee
and the constants of the equivalence depend on $r$, $D$, $k$ and $d$.
\end{example}
\begin{proof}
To prove~\eqref{e4} we use two affine maps of the cube, $D_1$ and $D_2$ which map $(1,\dots,1)$ to $\x$ such that $D_1\subset D\subset D_2$, and follow earlier considerations.

To prove~\eqref{e5} we first prove it on the cube $[-1,1]^d$. With no loss of generality we may choose the corresponding $S_k$ (a $k$-face of the cube) to be $S_k=\{(z_1,\dots,z_k,-1,\dots,-1):z_i\in[-1,1]\}$. Hence, for $\y=(y_1,\dots,y_d)$ satisfying $(\y+rB_2^d)\cap H\subset S_k$ (with $H$ of the definition of $S_k$) we have $y_i=-1$ for $i>k$, and $|y_i|\le1-r$ for $i\le k$.
The function $C(\P_n,[-1,1]^d,\x)$ which is given by $\sum_{m=1}^{\eta}\Phi_{m}(\x)^2$, where $\{\Phi_{m}\}_{m=1}^{\eta}$ is any orthonormal basis of the space of algebraic polynomials of total degree $\le n$ on $[-1,1]^d$ (of dimension $\eta={{n+d}\choose n}$), satisfies for any $\y\in\R^d$
\[
\prod_{l=1}^d\sum_{j=0}^{\lfloor n/d\rfloor}\phi_j(y_l)^2
\le
\sum_{m=1}^{\eta}\Phi_{m}(\y)^2\le\prod_{l=1}^d\sum_{j=0}^n\phi_j(y_l)^2
\]
where $\phi_j$ is the Legendre polynomial of degree $j$ such that $\|\phi_j\|_{L_2[-1,1]}=1$. We now note that as $|y_i|\le 1-r$ for $i\le k$, Observation~\ref{eo1} implies $\sum_{j=0}^n\phi_j(y_i)^2\approx \sum_{j=0}^{\lfloor n/d\rfloor}\phi_j(y_i)^2\approx n$. For $i>k$ we use Theorem~\ref{cube_general} to get $\sum_{j=0}^n\phi_j(-1)^2\approx \sum_{j=0}^{\lfloor n/d\rfloor}\phi_j(-1)^2\approx n^2$. The above implies~\eqref{e5} when $D$ is the cube $[-1,1]^d$ and, using Theorem~\ref{l3.5}, it implies~\eqref{e5} for any affine transformation of the cube.

To prove~\eqref{e5} for general polytopes, we define $D_1$ as an affine map of the cube with first $k$ coordinates (after the mapping) being in $S_k$. The length of the first $k$ coordinates of $D_1$ is $a$ and $\y$ is the center of that face. We then have the other mapped coordinates to be inside the polytope $D$. We now define $D_2$ as a map of the $d$-dimensional cube with the first $k$ coordinates parallel to those of $D_1$ but of length $2\diam(D)$ and with $\y$ in the center of that face of $D_2$. The other coordinates of $D_2$ are set to be perpendicular to the first $k$ with size $\diam(D)$ and on the side of $D$ (so that $D\subset D_2$). Earlier considerations now imply~\eqref{e5}.
\end{proof}

\begin{example}\label{ee4}
The cone $D$, the convex hull of $\bxi=(0,0,1)$ and $\{(x,y,z):z=0,x^2+y^2\le1\}=B_2^2\times\{0\}$, satisfies
\be\label{e6}
\sup_{\phi\in\P_n}|\phi(\bxi)|/\|\phi\|_{L_2(D)}\approx
\sup_{\phi\in\P_n}\|\phi\|_{L_\infty(D)}/\|\phi\|_{L_2(D)}\approx n^{3}.
\ee
\end{example}
\begin{proof}
We set $D_1$ as the convex hull of $(0,0,1)$, $(0,\pm1,0)$ and $(\pm1,0,0)$ which satisfies $D_1\subset D$ and is a pyramid which is a polytope in $\R^3$. We set $D_2$ as the convex hull of $(0,0,1)$, $(0,\pm\sqrt{2},0)$ and $(\pm\sqrt{2},0,0)$ which satisfies $D\subset D_2$, and $D_2$ is also a polytope (pyramid) in $\R^3$. Using~\eqref{e4} and earlier considerations, we have~\eqref{e6}.
\end{proof}

Obviously, the relation~\eqref{e6} holds for any affine map of $D$ in Example~\ref{ee4}.

\begin{example}\label{ee5}
For $D$ half the disc i.e. $D=\{(x,y):x^2+y^2\le1,\,y\ge0\}$
\be\label{e7}
\sup_{\phi\in\P_n}|\phi(\pm1,0)|/\|\phi\|_{L_2(D)}\approx
\sup_{\phi\in\P_n}\|\phi\|_{L_\infty(D)}/\|\phi\|_{L_2(D)}\approx n^{2}.
\ee
For $D$ quarter ball i.e. $D=\{(x,y,z):x^2+y^2+z^2\le1,\,y\ge0,\,z\ge0\}$
\be\label{e8}
\sup_{\phi\in\P_n}|\phi(\pm1,0,0)|/\|\phi\|_{L_2(D)}\approx
\sup_{\phi\in\P_n}\|\phi\|_{L_\infty(D)}/\|\phi\|_{L_2(D)}\approx n^{3}.
\ee
\end{example}
\begin{proof}
To prove~\eqref{e7} we set $D_2=\{(x,y):|x|\le1,\,0\le y\le1\}=[-1,1]\times[0,1]$ and $D_1=\{(x,y):0\le|x|+y\le1,\,y\ge0\}$.

To prove~\eqref{e8} we set $D_2=\{(x,y,z):|x|\le1,\,0\le y\le1,\,0\le z\le1\}=[-1,1]\times[0,1]\times[0,1]$ and
$D_1=\{(x,y):0\le|x|+y+z\le1,\,y\ge0,\,z\ge0\}$.

As $D_i$ are polytopes in $\R^2$ and $\R^3$, the results follow the usual considerations.
\end{proof}

For half the ball $D=\{(x,y,z):x^2+y^2+z^2\le1,\,z\ge0\}$ the situation is more elaborate and will be dealt with in Section~\ref{sect-sharp} where the corresponding behaviour will be shown to be $n^{5/2}$ ($\sigma=5$).

\begin{example}\label{ee6}
Suppose $D=A\times B=\{(\x,\y):\x\in A\subset\R^k,\,\y\in B\subset\R^l\}$ such that $\sup_{\phi\in\P_{n,k}}\|\phi\|_{L_\infty(A)}/\|\phi\|_{L_2(A)}\approx n^{\sigma_1/2}$ and
$\sup_{\phi\in\P_{n,l}}\|\phi\|_{L_\infty(B)}/\|\phi\|_{L_2(B)}\approx n^{\sigma_2/2}$ where
$\P_{n,s}$ is the set of polynomials of degree $n$ in $s$ variables. Then
\be\label{e9}
\sup_{\phi\in\P_{n,k+l}}\|\phi\|_{L_\infty(D)}/\|\phi\|_{L_2(D)}\approx n^{(\sigma_1+\sigma_2)/2}.
\ee
For the cylinder $D=\{(x,y,z):x^2+y^2\le1,\,0\le z\le1\}=B_2^2\times[0,1]$, the above implies
\[
\sup_{\phi\in\P_n}|\phi(\bxi)|/\|\phi\|_{L_2(D)}\approx
\sup_{\phi\in\P_n}\|\phi\|_{L_\infty(D)}/\|\phi\|_{L_2(D)}\approx n^{5/2},
\]
where $\bxi=(x,y,z)$ is any point which satisfies $x^2+y^2=1$ and $z=0$ or $z=1$.
\end{example}
\begin{proof}
Let $\{\Phi_{j,n}(\x,\y)\}$ be any orthonormal basis of polynomials (over $A\times B$) in $k+l$ variables which are of total degree $\le n$. The inverse of the Christoffel function $C(\P_n,A\times B, (\x,\y))=\sum_j\Phi_{j,n}(\x,\y)^2$ is independent of the choice of the orthonormal basis of $\P_n$. Therefore,
\be
C(\P_{n/2,k},A,\x)C(\P_{n/2,l},B,\y)\le C(\P_{n,k+l},A\times B,(\x,\y))\le C(\P_{n,k},A,\x) C(\P_{n,l},B,\y).
\ee
This implies~\eqref{e9} using Theorem~\ref{dt3.1}.

The implication on the cylinder follows, as on the disc, the supremum is achieved at any point of the unit circle, and on the interval $[0,1]$ the supremum is achieved at $0$ and $1$ (see Theorem~\ref{ball-nik} and Theorem~\ref{cube_general}).
\end{proof}

\begin{observation}\label{eo7}
For domain $D$ and affine transformation $T$ Theorem~\ref{l3.5} implies
\be\label{e10}
\sup_{\phi\in\P_n}|\phi(\bxi)|/\|\phi\|_{L_2(D)}=
|\det T|^{-1/2}\sup_{\phi\in\P_n}|\phi(T(\bxi))|/\|\phi\|_{L_2(TD)}
\ee
and for $0<p\le q\le \infty$ and $\phi\in\P_n$
\[
\|\phi\|_{L_q(D)}\le c(p,q)n^{\sigma(\frac1p-\frac1q)}\|\phi\|_{L_p(D)}
\quad\Longleftrightarrow\quad
\|\phi\|_{L_q(TD)}\le c(p,q)|\det T|^{\frac1p-\frac1q}n^{\sigma(\frac1p-\frac1q)}\|\phi\|_{L_p(TD)}.
\]
\end{observation}
Hence the transformation $T$ influences the constant in front of $n^{\sigma(\frac1p-\frac1q)}$ but not the $\sigma$ of the Nikol'skii inequality (unless $T$ depends on $n$). Moreover, if two points, $\bxi_1$ and $\bxi_2$, satisfy
\[
\sup_{\phi\in\P_n}|\phi(\bxi_1)|/\|\phi\|_{L_2(D)}=
\sup_{\phi\in\P_n}|\phi(\bxi_2)|/\|\phi\|_{L_2(D)},
\]
the equality holds for $T\bxi_1$ and $T\bxi_2$ in relation to $TD$ as well.
The influence of the above on ``thin'' sets can be illustrated by considering the triangle $A$ given by $(\pm1,0)$ and $(0,\sqrt{3})$ and a second triangle $B$ given by $(\pm1,0)$ and $(0,\eps)$ for which the transformation $TA=B$ yields
\[
\|\phi\|_{L_\infty(A)}\le cn\|\phi\|_{L_2(A)}
\quad\text{implies}\quad
\|\phi\|_{L_\infty(B)}\le c\left(\frac{\sqrt{3}}{\eps}\right)^{1/2}n\|\phi\|_{L_2(B)}.
\]
The equivalence~\eqref{e9} also yields the following: for $\bmu_1=(1,0)$, $\bmu_2=(-1,0)$ and $\bmu_3=(0,\eps)$, we have
\[
\sup_{\phi\in\P_n}|\phi(\bmu_i)|/\|\phi\|_{L_2(B)}=
\sup_{\phi\in\P_n}|\phi(\bmu_j)|/\|\phi\|_{L_2(B)}.
\]



\begin{theorem}\label{C2}
Suppose $D$ is a compact set in $\R^d$ such that the boundary $\partial D$ is a $(d-1)$-dimensional $C^1$ submanifold in $\R^d$ (in the sense of differential geometry), and suppose the outward pointing unit normal vector $\n(\x)$ for $\x\in \partial D$ satisfies the Lipschitz condition
\[
|\n(\x)-\n(\y)|\le L|\x-\y|, \quad \x,\y\in\partial D,
\]
with some $L>0$. Then $C(\P_n,D) \approx n^{d+1}$.
\end{theorem}
\begin{proof}
By the equivalence of (iii) and (v) in~\cite[Theorem~1]{Walther}, a ball of radius $L^{-1}$ rolls freely inside $D$ (in the terminology of~\cite{Walther}). This immediately implies that for some family of centres $\{\x_\lambda\}_\lambda\subset D$, our set $D$ is the union of the Euclidean balls of radius $L^{-1}$ with those centres, i.e., ${D}=\bigcup_\lambda (\x_\lambda+L^{-1} B_2^d)$ and, using Theorem~\ref{dt3.5}, we complete the proof.
\end{proof}

We note that if $D$ is a compact set in $\R^d$ such that the boundary $\partial D$ is a $(d-1)$-dimensional $C^2$ submanifold in $\R^d$, then the Lipschitz condition is satisfied, and the above theorem is applicable. This will be applied in the next section for $l_\alpha$ balls in $\R^d$ with $2\le \alpha<\infty$.

Let us give some other examples. For a torus in $\R^3$ given by
\[
D:=\left\{(x,y,z):\left(R-\sqrt{x^2+y^2}\right)^2+z^2\le r^2\right\}, \quad 0<r<R,
\]
we have $C(\P_n,D)\approx n^4$. For any convex body $D$ in $\R^2$ with $C^2$ boundary we get $C(\P_n,D)\approx n^3$. Moreover, the same is true even if the convexity requirement is dropped, for example, for smooth ``flower''-like domains such as $D=\{(r\cos\theta,r\sin\theta): 0\le r\le 2+\sin(10\theta)\}$.

\section{Computation of the Christoffel function on $l_\alpha$ balls in $\R^d$ for $1\le \alpha\le \infty$}\label{sect-lp}

We introduce the following notation for unit balls in $\R^d$ with respect to the $l_\alpha$ metric:
\[
B^d_\alpha:=\{\x\in\R^d:\|\x\|_\alpha=(|x_1|^\alpha+\dots+|x_d|^\alpha)^{\frac1\alpha}\le1\}, \quad 1\le \alpha<\infty,
\]
and when $\alpha=\infty$:
\[
B^d_\infty:=\{\x\in\R^d:\|\x\|_\infty=\max\{|x_1|,\dots,|x_d|\}\le1\}=[-1,1]^d.
\]

Using the results from Sections~\ref{sect-upper}, \ref{sect-lower} and~\ref{sect-examples}, we will prove the following theorem.
\begin{theorem}\label{balls_p_le2}
For any $d\ge 2$
\[
C(\P_n,B^d_\alpha)\approx n^\sigma
\quad\text{and}\quad
\sup_{\phi\in\P_n}\|\phi\|_{L_\infty(B^d_\alpha)}/\|\phi\|_{L_2(B^d_\alpha)}\approx n^{\sigma/2},
\]
where
\[
\sigma = \begin{cases} {2+\frac2\alpha(d-1)}, &\text{if}\quad 1\le \alpha\le 2,  \\
{d+1}, &\text{if}\quad 2\le \alpha<\infty, \\
{2d}, &\text{if}\quad \alpha=\infty.
\end{cases}
\]
\end{theorem}

Using Theorem~\ref{dt3.4} and Theorem~\ref{cube-in},  we have $c_1 n^{d+1}\le C(\P_n,D)\le c_2 n^{2d}$ for any convex body $D$ in $\R^d$. Theorem~\ref{balls_p_le2} shows that the full range of powers $\sigma$ between $d+1$ and $2d$ in $C(\P_n,D)\approx n^\sigma$ can be attained for convex bodies $D$ in $\R^d$.

The non-trivial case in Theorem~\ref{balls_p_le2} is $1<\alpha<2$ due to the fact that the boundary of $B^d_\alpha$ is rather ``narrow'' near the point $(1,0,\dots,0)$ or $(0,\dots,0,1,0,\dots,0)$, and so one cannot inscribe a Euclidean ball of any radius in $B^d_\alpha$ containing the point $(1,0,\dots,0)$. The following two lemmas are needed for the upper estimate in the case $1<\alpha<2$.

\begin{lemma}\label{abp_est}
For any $a\ge 0$, $b\in\R$, and $1\le \alpha\le 2$, the following inequality holds:
\[
|a+b|^\alpha-a^\alpha-\alpha a^{\alpha-1}b\le 7|b|^\alpha.
\]
\end{lemma}
\begin{proof}
If $a\le \frac87|b|$, then
\[
\textstyle |a+b|^\alpha-a^\alpha-\alpha a^{\alpha-1}b\le (\frac{15}7|b|)^\alpha+ 0+\alpha(\frac87)^{\alpha-1}|b|^\alpha\le((\frac{15}7)^2+2(\frac87))\le 7|b|^\alpha.
\]
Otherwise, $|b|<\frac78 a$, in particular, $a+b>a/8>0$. Consider $f(t)=t^\alpha$, $t>0$. For some $\lambda$ between $a+b$ and $a$ we have
\[
f(a+b)=f(a)+f'(a)b+\frac{f''(\lambda)}2b^2.
\]
As $\lambda>\min\{a+b,a\}>a/8>|b|/7$ and $-1\le\alpha-2\le 0$, we obtain
\[
|a+b|^\alpha-a^\alpha-\alpha a^{\alpha-1}b=\frac{\alpha(\alpha-1)}{2}\lambda^{\alpha-2}b^2\le\frac{\alpha(\alpha-1)}{2}\left(\frac{|b|}7\right)^{\alpha-2}|b|^2\le 7|b|^\alpha.
\]
\end{proof}



\begin{lemma}\label{cone_ball_in_p_le_2}
Suppose $1< \alpha\le 2$, $d\ge 2$, and $\x\in\partial B^d_\alpha$, and suppose further that $\u$ is the outward normal vector for the surface $\partial B^d_\alpha$ at the point $\x$ with $\v$ any unit vector from $\R^d$ orthogonal to $\u$. Then, for any $\delta\in(0,\beta_1)$
\[
\x-\delta \u + \beta_2 \delta^{1/\alpha} \v \in B^d_\alpha,
\]
where $\beta_2=\beta_2(d)>0$ depends only on $d$, and $\beta_1=\beta_1(d,\alpha)>0$ depends only on $d$ and $\alpha$. (The proof yields that $\lim_{\alpha\to1+}\beta_1(d,\alpha)=0$.)
\end{lemma}
\begin{proof}
The components of $\x$, $\u$, and $\v$ will be denoted $x_j$, $u_j$, and $v_j$ respectively, $j=1,\dots,d$. Without loss of generality we assume $x_j\ge0$, $j=1,\dots,d$. It is straightforward that
\[
u_j=\gamma x_j^{\alpha-1}, \quad \gamma = (x_1^{2(\alpha-1)}+\dots+x_d^{2(\alpha-1)})^{-1/2}.
\]
There exists $j$ such that $x_j\ge 1/d$, and hence, $\gamma\le d^{\alpha-1}=:\gamma_1$. We choose
$\beta_2:=\frac12(7 d \gamma_1)^{-1/2}$, which guarantees
\begin{equation}\label{beta_est}
-\alpha\gamma^{-1}+7d(2\beta_2)^\alpha< -\gamma_1^{-1}+7d(7d\gamma_1)^{-\alpha/2}\le0.
\end{equation}
Further, selecting $\beta_1:=(\beta_2\gamma_1^{-1})^{\alpha/(\alpha-1)}$ ensures $\delta\gamma<\beta_2\delta^{1/\alpha}$, and hence
\begin{equation}\label{alpha_est}
|-\delta\gamma x_j^{\alpha-1}+\beta_2\delta^{1/\alpha}v_j|
\le \delta\gamma+\beta_2\delta^{1/\alpha}<2\beta_2\delta^{1/\alpha}.
\end{equation}
Using Lemma~\ref{abp_est}, \eqref{alpha_est}, and~\eqref{beta_est}, we conclude
\begin{align*}
\sum_{j=1}^d|x_j-\delta u_j+\beta_2\delta^{1/\alpha}v_j|^\alpha &
\le \sum_{j=1}^d\left(x_j^\alpha+\alpha x_j^{\alpha-1}(-\delta u_j+\beta_2\delta^{1/\alpha}v_j)+7 \left|-\delta u_j+\beta_2\delta^{1/\alpha}v_j\right|^\alpha\right)\\
&= \sum_{j=1}^d x_j^\alpha - \alpha \gamma \delta \sum_{j=1}^d x_j^{2(\alpha-1)} + \beta_2 \alpha\delta^{1/\alpha}\gamma^{-1} \sum_{j=1}^d u_jv_j \\
&\phantom{xxxxxxxxxxxxxxxxxxxxxxxxx}+7 \sum_{j=1}^d \left|-\delta \gamma x_j^{\alpha-1}+\beta_2\delta^{1/\alpha}v_j\right|^\alpha\\
&=1- \alpha \gamma^{-1} \delta +0+7 \sum_{j=1}^d \left|-\delta \gamma x_j^{\alpha-1}+\beta_2\delta^{1/\alpha}v_j\right|^\alpha\\
&\le 1 - \alpha \gamma^{-1} \delta+7d(2\beta_2)^\alpha\delta\le1,
\end{align*}
which completes the proof.
\end{proof}

\begin{proof}[Proof of Theorem~\ref{balls_p_le2}.]
If $\alpha=1$ or $\alpha=\infty$, then $B^d_\alpha$ is a polytope, so the result follows from Example~\ref{ee3}. For $2\le \alpha<\infty$, since the univariate function $t\mapsto |t|^\alpha$ is twice continuously differentiable, we have that the function $(x_1,\dots,x_{d-1})\mapsto(1-|x_1|^\alpha-\dots-|x_{d-1}|^\alpha)^{1/\alpha}$ is twice continuously differentiable on the interior of $B^{d-1}_\alpha$. Hence, the boundary of $B^d_\alpha$ is a $C^2$ surface, so we can simply use Theorem~\ref{C2}.

Now assume $1<\alpha<2$. For the upper estimate, Theorem~\ref{boundary_const}, Lemma~\ref{cone_ball_in_p_le_2} and Theorem~\ref{cone-set-ell} with $s=\frac1\alpha$ imply
\[
C(\P_n,B^d_\alpha)\le c n^{2+\frac2\alpha(d-1)}.
\]
For the lower estimate, we will use Theorem~\ref{parallel}. First we compute and estimate the parallel section function of $B^d_\alpha$ for an appropriate $\bxi$. Note that for fixed $a\in(-1,1)$, the section
\[
B^d_\alpha\cap\{\x\in\R^d:x_1=a\}
\]
is exactly
\[
(1-|a|^\alpha)^{1/\alpha} B^{d-1}_\alpha
\]
in $\R^{d-1}$. Therefore, for $t\in(0,2)$, $a=1-t$, $\bxi=(-1,0,\dots,0)$, we have
\[
A_{B^d_\alpha,\bxi}(t)=(1-|1-t|^\alpha)^{(d-1)/\alpha}\,\Vol_{d-1}(B^{d-1}_\alpha).
\]
As $1-|1-t|^\alpha\le \alpha\,t$, we obtain $A_{B^d_\alpha,\bxi}(t)\le M t^{(d-1)/\alpha}$ for any $t>0$ with some $M>0$ depending only on $d$ and $\alpha$. Theorem~\ref{parallel} with $\lambda=(d-1)/\alpha$ provides
\[
C(\P_n,B^d_\alpha)\ge c n^{2+\frac2\alpha(d-1)},
\]
as required.
\end{proof}


\begin{remark}
In fact, the lower estimate in the above proof implies that
\[
C(\P_n,B^d_\alpha)\ge c n^{2+\frac2\alpha(d-1)}, \quad 0<\alpha<1.
\]
For $0<\alpha<1$ the set $B^d_\alpha$ is not convex (see also Theorem~\ref{cube-in}). The ``cusps'' of $B^d_\alpha$ cause that the power $2+\frac2\alpha(d-1)\to\infty$ as $\alpha\to0+$.
\end{remark}

\begin{remark}
Remez-type inequalities can be used to derive Nikol'skii-type inequalities, see, e.g. the general framework in~\cite{DiPrRMMJ}. In particular, it is possible to obtain appropriate upper bounds on $C(\P_n,B^d_\alpha)$, $1<\alpha<2$, from the Remez inequality given in~\cite{Kr}, if one establishes that $B^d_\alpha$ is a $C^\alpha$-domain in terminology of~\cite{Kr}. Note that in the present paper we provide not only the matching lower bounds, but also introduce the ``ellipsoid with extension'' technique of Section~\ref{sect-upper} for upper bounds which works in situations where the $C^\alpha$-domain classification of~\cite{Kr} will lead to a much weaker estimate. A non-trivial example of such a situation is given in the next section.
\end{remark}

\section{Sharp subsets and further examples}\label{sect-sharp}
We will present an improvement of Theorem~\ref{boundary_const} and illustrate its application.
\begin{definition}\label{sharp_def}
For a compact set $D\subset\R^d$, a set $\Omega\subset D$ is a \emph{sharp subset} of $D$ if there exists a constant $c>0$ such that for any $\x\in D$ there exists an affine transformation $T$ with  $|\det T|\ge c$, satisfying
\[
 \x\in T(\Omega) \quad\text{and}\quad T(D)\subset D.
\]
\end{definition}
Trivially, $D$ is always a sharp subset of $D$ with $c=1$ as it is enough to take $T$ to be the identity. We are, however, interested in rather ``small'' sharp subsets. For instance, it is not hard to see that any vertex of a convex polytope is a sharp subset of that polytope. Another example was given in Lemma~\ref{boundary_conv_lemma}, which states that for any convex body $D$ the boundary $\partial D$ is a sharp subset of $D$ with $c=2^{-d}$. Repeating the steps of the proof of Theorem~\ref{boundary_const}, we obtain the following:
\begin{theorem}\label{sharp_th}
Let $\Omega$ be a sharp subset of a compact set $D\subset\R^d$. Then
\[
\max_{\x\in D} C(\P_n, D, \x)\le c^{-1} \max_{\x\in \Omega} C(\P_n,D,\x),
\]
where $c$ is the constant from Definition~\ref{sharp_def}.
\end{theorem}
Therefore, if we established that $\Omega$ is a sharp subset of $D$, then to compute the order of $\max_{\x\in D} C(\P_n,D,\x)$, it is sufficient to consider only $\x\in \Omega$.
\begin{remark}\label{boundary_only}
If $D$ is a convex body, then Lemma~\ref{boundary_conv_lemma} allows one to restrict verification of the condition from Definition~\ref{sharp_def} to $\x\in\partial D$. Indeed, suppose for any $\x\in\partial D$, there exists $T$ with $|\det T|\ge c>0$, $\x\in T(\Omega)$ and $T(D)\subset D$. Take any $\tx\in D$ and apply Lemma~\ref{boundary_conv_lemma} to obtain an affine transformation $T_1$ with $|\det T_1|\ge \frac1{2^d}$, $\tx\in T_1(\partial D)$ and $T_1(D)\subset D$. We take $\x:=T_1^{-1}(\tx)\in\partial D$, obtain the corresponding $T$, and consider the composition $T_2(\cdot):=T_1(T(\cdot))$. Then $|\det T_2|\ge \frac c{2^d}$, $\tx\in T_2(\Omega)$ and $T_2(D)\subset D$, and therefore $\Omega$ is a sharp subset of $D$ with the constant $\frac{c}{2^d}$.
\end{remark}

Let us introduce a notation for half the Euclidean ball:
\[
B^d_+:=\{(x_1,\dots,x_d)\in B^d_2: x_d\ge0\}.
\]
We have already considered $B^2_+$ in~\eqref{e7}. Now we illustrate the application of Theorem~\ref{sharp_th} for $B^3_+$.

\begin{lemma}\label{sharp_b3p}
The set $\{(1,0,0)\}$ is a sharp subset of $B^3_+$.
\end{lemma}
\begin{proof}
Taking into account Remark~\ref{boundary_only}, we only consider $\x=(x_1,x_2,x_3)\in\partial B^3_+$.

First we suppose that $x_3=0$. Then, by symmetry (rotating about the $x_3$-axis), we can assume $x_2=0$ and $x_1\in[0,1]$. We define $T(\cdot)=(-1,0,0)+\frac{1+x_1}{2}((\cdot)-(-1,0,0))$, which is the homothety with coefficient $\frac{1+x_1}{2}\in[\frac12,1]$ and center at $(-1,0,0)$. Clearly, $T(1,0,0)=\x$ and $T(B^3_+)\subset B^3_+$, $c\ge\frac18$.

It remains to consider the case when $x_3>0$. As before, by symmetry (rotating about $x_3$-axis), we can assume $x_2=0$ and $x_1=\sqrt{1-x_3^2}\in[0,1)$. Then $x_1=\cos\phi$ and $x_3=\sin\phi$ for some $\phi\in(0,\frac\pi2]$. We define $T=RH$, where $H(\cdot)=(1,0,0)+\frac12((\cdot)-(1,0,0))$ is the homothety with the coefficient $\frac12$ and center at $(1,0,0)$, and $R(y_1,y_2,y_3)=(y_1\cos\phi-y_3\sin\phi,0,y_1\sin\phi+y_3\cos\phi)$ is the rotation counterclockwise by $\phi$ in the $y_1y_3$-plane. It is easy to see that $T(1,0,0)=\x$ and $T(B^3_+)\subset B^3_+$ (as $\phi\in(0,\pi/2]$), $c=\frac18$.
\end{proof}

\begin{remark}
It is natural to expect that $\{(1,0,\dots,0)\}$ is a sharp subset of $B^d_\alpha$, $1<\alpha<2$, but we could neither prove this fact nor locate it in the literature. Computations in Lemma~\ref{cone_ball_in_p_le_2} would have been much simpler if we only needed to consider $\x=\{(1,0,\dots,0)\}$. However, establishing that $\{(1,0,\dots,0)\}$ is a sharp subset of $B^d_\alpha$, $1<\alpha<2$, may, in fact, turn out to be much harder than directly proving Lemma~\ref{cone_ball_in_p_le_2}.
\end{remark}

We continue with $B^3_+$. We can inscribe an appropriate ellipsoid with extension into $B^3_+$, as shown below.
\begin{lemma}
Let $D=\{(x_1,x_2,x_3):x_1^2+x_2^2+x_3^2\le1,x_3\ge0\}$ be a half ball in $\R^3$. For every $n\ge1$, there exists an affine transformation $T$ on $\R^3$ with $|\det T|\ge (160n)^{-1}$, such that $TB_2^3\subset D$ and $T\v_n^3=(1,0,0)$.
\end{lemma}
\begin{proof}
We define $T$ as follows:
\[
T(x_1,x_2,x_3)=\left(1-\frac{1+\frac{n^{-2}}3-x_1}2, \frac{x_2}8, \frac{n^{-1}x_3}{10}+\frac{1+\frac{n^{-2}}3-x_1}8\right).
\]
Then $\det A=\frac12\cdot\frac18\cdot\frac{n^{-1}}{10}$, and clearly $T\v_n^3=T(1+n^{-2}/3,0,0)=(1,0,0)$. By convexity of $D$, it remains to show that $TB^3_2\subset D$. Let $(x_1,x_2,x_3)\in B^3_2$, i.e., $x_1^2+x_2^2+x_3^2\le 1$. Denote $(y_1,y_2,y_3)=T(x_1,x_2,x_3)$, and then we need to prove $y_3\ge0$ and $y_1^2+y_2^2+y_3^2\le1$. Let $t:=1-x_1\in[0,2]$. For $j=2,3$, we obtain
\[
|x_j|\le\sqrt{1-x_1^2}=\sqrt{2t-t^2}\le\sqrt{2t}.
\]
We also have
\[
\frac{1+\frac{n^{-2}}3-x_1}{8}=\frac{\frac{n^{-2}}3+t}{8}\ge\frac{\sqrt{\frac{n^{-2}}3t}}4=\frac{n^{-1}\sqrt{2t}}{4\sqrt{6}}\ge\frac{n^{-1}\sqrt{2t}}{10}.
\]
The above two inequalities yield
\[
0\le -\frac{n^{-1}|x_3|}{10} + \frac{n^{-1}\sqrt{2t}}{10} \le y_3 \le \frac{n^{-1}|x_3|}{10} + \frac{\frac{n^{-2}}3+t}{8} < \frac{\frac{n^{-2}}3+t}{4},
\]
in particular, $y_3\ge0$. The remaining inequality is obtained as follows:
\begin{align*}
y_1^2+y_2^2+y_3^2 -1&=-\left(\frac{n^{-2}}3+t\right)+\frac14\left(\frac{n^{-2}}3+t\right)^2+\left(\frac{x_2}8\right)^2+y_3^2\\
&\le-\frac{n^{-2}}3-t+\frac14\left(\frac{n^{-2}}3+t\right)^2+\frac{t}{32}+\frac1{16}\left(\frac{n^{-2}}3+t\right)^2\\
&\le -\frac{n^{-2}}3-\frac{31}{32}t+\frac{15}{16}\left(\frac{n^{-2}}3+t\right)=-\frac{n^{-2}}{48}-\frac{t}{32}<0,
\end{align*}
where in the last line $\frac{n^{-2}}3+t<3$ was used. This completes the proof.
\end{proof}

Now we are ready to compute the order of $C(\P_n,B^3_+)$.
\begin{theorem}
\[
C(\P_n,B^3_+)\approx n^5
\quad\text{and}\quad
\sup_{\phi\in\P_n}\|\phi\|_{L_\infty(B^3_+)}/\|\phi\|_{L_2(B^3_+)}\approx n^{5/2}.
\]
\end{theorem}
\begin{proof}
The previous lemma, Lemma~\ref{sharp_b3p}, Theorem~\ref{sharp_th}, and Theorem~\ref{main-ell}  immediately imply $C(\P_n,B^3_+)\le c n^5$. For the lower estimate, use Theorem~\ref{main-lower} with $T(x_1,x_2,x_3)=(x_1,x_2,x_3+1)$ and $\y=(1,0,-1)$. Then $C(\P_n,B^3_+)^{-1}\le c \rho_n(1)\rho_n(0)\rho_n(-1)\le \tilde c n^{-5}$.
\end{proof}

\begin{remark}
It is not hard to generalize the above approach to show that $C(\P_n,B^d_+)\approx n^{d+2}$ for $d\ge4$. To avoid technicalities we decided to restrict the exposition to $d=3$, as this case is sufficient to illustrate the method.
\end{remark}

\end{document}